\documentclass[12pt, a4paper]{article}
\usepackage{}

\usepackage{amsfonts}
\usepackage{bbm}
\usepackage{mathrsfs}
\usepackage{latexsym}
\usepackage{graphicx}
\usepackage{amssymb}
\usepackage{amsmath}
\usepackage{color,xcolor}
\usepackage{mathtools}

\newtheorem{theorem}{Theorem}[section]
\newtheorem{lemma}[theorem]{Lemma}
\newtheorem{corollary}[theorem]{Corollary}

\newtheorem{problem}[theorem]{Problem}

\newtheorem{conjecture}[theorem]{Conjecture}

\def\det{{\rm det}}

\title{{\Large \bf On the sum of the largest $A_{\alpha}$-eigenvalues of graphs
\thanks{ Supported by the National Natural Science Foundation of China (No. 11771443, 12071411).}~}}
\author{Zhen Lin$^{a}$\thanks{Corresponding author. E-mail addresses: lnlinzhen@163.com(Z. Lin),
miaolianying@cumt.edu.cn (L. Miao), ychgsg@163.com (S.-G. Guo).}, Lianying Miao$^a$, Shu-Guang Guo$^b$\\
{\footnotesize $^a$School of Mathematics, China University of Mining and Technology,}\\ {\footnotesize  Xuzhou, 221116, Jiangsu, P.R.
China}\\
\footnotesize  $^b$School of Mathematics and Statistics, Yancheng Teachers
University,
\\  \footnotesize  Yancheng,  224002, Jiangsu, P.R. China }

\date{}
\begin{document}
\openup 1.0\jot
\date{}\maketitle
\begin{abstract}
For every real $0\leq \alpha \leq 1$, Nikiforov defined the $A_{\alpha}$-matrix of a graph $G$ as $A_{\alpha}(G)=\alpha D(G)+(1-\alpha)A(G)$, where $A(G)$ and $D(G)$ are the adjacency matrix and the degree diagonal matrix of a graph $G$, respectively. The eigenvalues of $A_{\alpha}(G)$ are called the $A_{\alpha}$-eigenvalues of $G$. Let $S_k(A_{\alpha}(G))$ be the sum of $k$ largest $A_{\alpha}$-eigenvalues of $G$. In this paper, we present several upper and lower bounds on $S_k(A_{\alpha}(G))$ and characterize the extremal graphs for certain cases, which can be regard as a common generalization of the sum of $k$ largest eigenvalues of adjacency matrix and signless Laplacian matrix of graphs. In addition, some graph operations on $S_k(A_{\alpha}(G))$ are presented.

\bigskip


\end{abstract}
\baselineskip 20pt

\section{\large Introduction}

\ \ \ \ Let $G$ be a simple undirected graph with vertex set $V(G)$ and edge set $E(G)$. Denote by $K_n$, $P_n$, $C_n$ and $K_{1,\, n-1}$ the complete graph, path, cycle and star with $n$ vertices, respectively. Let $d_v=d_G(v)$ be the degree of vertex $v$ of the graph $G$. The minimum and maximum degree of a vertex in $G$ are denoted by $\delta=\delta(G)$ and $\Delta=\Delta(G)$, respectively. For a graph $G$, the first Zagreb index $Z_1=Z_1(G)$ is defined as the sum of the squares of the vertices degrees. There is a wealth of literature relating to the first Zagreb index, see for example \cite{BDFG, FT} and the references therein.

Let $\lambda_1(M)\geq \lambda_2(M)\geq \cdots \geq \lambda_n(M)$ be the eigenvalues of the real symmetric matrix $M$. Let $S_k(M)$ be the sum of $k$ largest eigenvalues of $M$. The investigation on the sum of $k$ largest eigenvalues of a real symmetric matrix is a topic of interest in matrix theory. The following classical theorem is due to Fan \cite{F}.

\begin{theorem}{\bf (\cite{F})}\label{th1,1} 
Let $M$ and $N$ be two real symmetric matrices of order $n$. Then
$$\sum\limits_{i=1}^{k}\lambda_i(M+N)\leq \sum\limits_{i=1}^{k}\lambda_i(M)+\sum\limits_{i=1}^{k}\lambda_i(N)$$
for any $1\leq k \leq n$.
\end{theorem}
Rojo et al. \cite{RSR} obtained some upper bounds for the sum of the $k$ largest eigenvalues of the matrix $M$ in terms of the trace of $M$. Mohar \cite{M}
showed that $S_k(M)$ is at most $\frac{1}{2}(\sqrt{k}+1)n$ when the entries of $M$ are between $0$ and $1$. Meanwhile, he gave an upper bound on the sum of the $k$ largest eigenvalues of arbitrary symmetric matrices. Nikiforov \cite{N1} obtained strengthen the upper bound and extend it to arbitrary $(0, 1)$-matrices.

Let $A(G)$ be the adjacency matrix of a graph $G$. For a graph $G$, Mohar \cite{M} showed that $S_k(A(G))$ is at most $\frac{1}{2}(\sqrt{k}+1)n$. This bound is shown to be best possible in the sense that for every $k$ there exist graphs whose sum is $\frac{1}{2}(\sqrt{k}+\frac{1}{2})n-o(k^{-2/5})n$. Das et al. \cite{DMS} proved an upper bound on $S_k(A(G))$ in terms of vertex number and negative inertia index. Moreover, Gernert \cite{G} showed that $S_2(A(G))\leq n$ if $G$ is a regular graph with $n$ vertices. He conjectured that this inequality holds for all
graphs. Gernert's conjecture was disproved by Nikiforov \cite{N2}, who gave examples of graphs with $S_2(A(G))\geq \frac{29+\sqrt{329}}{42}n-25>1.122n-25$ and proved that $S_2(A(G))\leq \frac{2}{\sqrt{3}}n<1.155n$. Ebrahimi et al. \cite{EMNA} showed that $S_2(A(G))\leq (\frac{1}{2}+\sqrt{\frac{5}{12}})n<1.145n$.

Let $Q(G)$ be the signless Laplacian matrix of a graph $G$. Ashraf et al. \cite{AOT} proposed the following conjecture on $S_{k}(Q(G))$.

\begin{conjecture}{\bf (\cite{AOT})}\label{co1,1} 
Let $G$ be a graph with $n$ vertices and $e(G)$ edges. Then
$$S_k(Q(G))\leq e(G)+\binom{k+1}{2}$$
for $1\leq k \leq n$.
\end{conjecture}
This conjecture has been proved to be correct for all graphs with at most ten vertices \cite{AOT}, all graphs with $k=1, 2, n-2, n-1, n$ \cite{AOT, CHJL}, regular graphs \cite{AOT}, trees \cite{HMT}, unicyclic graphs \cite{YY}, bicyclic graphs \cite{YY}, tricyclic graphs when $k\neq 3$ \cite{YY} and so on. Later,
Amaro et al. \cite{ALOLA} presented a strongly conjecture as follows.

\begin{conjecture}{\bf (\cite{ALOLA})}\label{co1,2} 
Let $G$ be a graph with $n\geq 5$ vertices and $3\leq k \leq n-2$ edges. Then
$$S_k(Q(G))\leq S_k(Q(H_{n,\, k}))<e(G)+\binom{k+1}{2}$$
with equality if and only if $G=H_{n,\, k}$, where $H_{n,\, k}$ is the $P_3$-join graph isomorphic
to $P_3[(n-k-1)K_1,K_{k-1}, K_2]$ for $3\leq k \leq n-2$.
\end{conjecture}
Moreover, Oliveira et al. \cite{OLRC} showed that the inequality $S_2(Q(G))\leq e(G)+3$ is tighter for the graph $K_{1,\,n-1}^{+}$ among all firefly graphs,
where $K_{1,\,n-1}^{+}$ is the star graph with an additional edge. Meanwhile, they conjectured that $K_{1,\,n-1}^{+}$ minimizes $f(G)=e(G)-S_2(Q(G))$ among all graphs $G$ with $n$ vertices. Recently, Du \cite{D} proved that $S_2(Q(G))<e(G)+3-\frac{2}{n}$ when $G$ is a tree, or a unicyclic graph whose unique cycle is not a triangle. This implies that the conjecture of Oliveira et al. is true for trees and unicyclic graphs whose unique cycle is not a triangle. Oliveira and Lima \cite{OL} showed that $S_2(Q(G))\geq d_1+d_2+1$ with equality if and only if $G$ is the star $K_{1,\,n-1}$ or the complete graph $K_3$, where $d_i$ is the $i$-largest degree of a vertex of $G$.

Another motivation to study $S_k(A(G))$ and $S_k(Q(G))$ came from the energy $\varepsilon(A(G))$ and signless Laplacian energy $\varepsilon(Q(G))$ of a graph $G$, which is very popular in mathematical chemistry. Let $G$ be a graph with $n$ vertices and $m$ edges. Then
$$\varepsilon(A(G))=\sum\limits_{k=1}^{n}\left\lvert \lambda_k(A(G)) \right\rvert=2S_{\sigma}(A(G))=\max\limits_{1\leq k \leq n}\left\{2S_{k}(A(G))\right\}$$
and
$$\varepsilon(Q(G))=\sum\limits_{k=1}^{n}\left\lvert \lambda_k(Q(G))-\frac{2m}{n} \right\rvert=2S_{\sigma}(Q(G))-\frac{4\sigma m}{n}=\max\limits_{1\leq k \leq n}\left\{2S_{k}(Q(G))-\frac{4km}{n}\right\},$$
where $\sigma$ denotes the number of the eigenvalues of $M$ greater than or equal to $tr(M)/n$. Thus $S_k(A(G))$ and $S_k(Q(G))$ are close relation with the energy and signless Laplacian energy, respectively. For more details in this field, we refer the reader to \cite{ACGMR, DMS, GCP, LSG}. In addition, $S_k(A(G))$ is related to Ky Fan norms of graphs introduced by Nikiforov \cite{N1}, which is a fundamental matrix parameter anyway.

For any real $\alpha \in[0,1]$, Nikiforov \cite{N} defined the matrix $A_{\alpha}(G)$ as
$$A_{\alpha}(G)=\alpha D(G)+(1-\alpha)A(G),$$
where $D(G)$ is the diagonal matrix of its vertex degrees and $A(G)$ is the adjacency matrix. It is easy to see that $A_{0}(G)=A(G)$ and $2A_{1/2}(G)=Q(G)$. The new matrix $A_{\alpha}(G)$ not only can underpin a unified theory of $A(G)$ and $Q(G)$, but it also brings many new interesting problems, see for example \cite{LDSS, LLX, LMG, N, NR}. This matrix has recently attracted the attention of many researchers, and there are several research papers published continually, see for example [8, 9, 19, 21, 23-28, 30-32, 34, 37, 38] and the references therein.

Motivated by the above works, we study the sum of $k$ largest eigenvalues of $A_{\alpha}(G)$. Since $S_k(A_{0}(G))=S_k(A(G))$ and $2S_k(A_{1/2}(G))=S_k(Q(G))$, $S_k(A_{\alpha}(G))$ can be regard as a common generalization of $S_k(A(G))$ and $S_k(Q(G))$. Moreover, if $G$ is a graph with $n$ vertices and $m$ edges, then
$$\varepsilon_{\alpha}(G)=\sum\limits_{k=1}^{n}\left\lvert \lambda_k(A_{\alpha}(G))-\frac{2\alpha m}{n} \right\rvert=2S_{\sigma}(A_{\alpha}(G))-\frac{4\alpha \sigma m}{n}=\max\limits_{1\leq k \leq n}\left\{2S_{k}(A_{\alpha}(G))-\frac{4\alpha k m}{n}\right\},$$
where $\varepsilon_{\alpha}(G)$ is the $\alpha$-energy of $G$ defined by Guo and Zhou \cite{GZ}. Thus $S_k(A_{\alpha}(G))$ is close relation with the $\alpha$-energy of $G$. In this paper, we obtain some upper and lower bounds on the sum of $k$ largest eigenvalues of $A_{\alpha}(G)$, which extend the results of $S_k(A(G))$ and $S_k(Q(G))$. In particular, the following problems and conjecture are proposed, repectively.

\begin{problem}\label{pro1,1} 
For a given $k$, which graph(s) minimize (or maximize) the sum of $k$ largest eigenvalues of $A_{\alpha}(G)$ among all graphs with $n$ vertices?
\end{problem}

\begin{conjecture}\label{co1,3} 
Let $G$ be a graph with $n$ vertices and $e(G)$ edges. If $\frac{1}{2}\leq \alpha < 1$, then
$$S_k(A_{\alpha}(G))\leq \alpha e(G)+\alpha\binom{k+1}{2}$$
for $1\leq k \leq n$.
\end{conjecture}

\begin{problem}\label{pro1,2} 
Which graph(s) minimize $f(G)=\alpha e(G)+\alpha +1-S_2(A_{\alpha}(G))$ for $\frac{1}{2}\leq \alpha< 1$?
\end{problem}

The remainder of this paper is organized as follows. In Section 2, we recall some useful notions and lemmas used further.  In Section 3, some upper bounds on $S_k(A_{\alpha}(G))$ are obtained. In Section 4, some upper bounds on the sum of the $k$ largest $A_{\alpha}$-eigenvalues of a tree are presented. In Section 5, some lower bounds on $S_k(A_{\alpha}(G))$ are given. Moreover, we prove that path is the minimum $S_2(A_{\alpha}(G))$ among all connected graphs for $\frac{1}{2}\leq \alpha< 1$, which is concerned with Problem \ref{pro1,1}. In Section 6, some graph operations on $S_k(A_{\alpha}(G))$ are presented.

\section{\large  Preliminaries}

 Let $\overline{G}$ be the complement of a graph $G$. The line graph $\mathcal{L}(G)$ is the graph whose vertex set are the edges in $G$, where two vertices are adjacent if the corresponding edges in $G$ have a common vertex. The $k$-th power $G^k$ of a graph $G$ is a graph with the same set of vertices as $G$ such that two vertices are adjacent in $G^k$ if and only if their distance in $G$ is at most $k$. The double graph $\mathcal{D}(G)$ of $G$ is a graph obtained by taking two copies of $G$ and joining each vertex in one copy with the neighbors of corresponding vertex in another copy. A clique of a graph $G$ is the maximal complete subgraph of the graph $G$. The independence number of $G$ is the maximum size of a subset of vertices of $G$ that contains no edge. A matching $\mathcal{M}$  of $G$ is a subset of $E(G)$ such that no two edges in $\mathcal{M}$ share a common vertex. The matching number of $G$ is the maximum number of edges of a matching in $G$. The chromatic number of a graph $G$ is the minimum number of colors such that $G$ can be colored in a way such that no two adjacent vertices have the same color. The nullity of a graph is the multiplicity of the eigenvalue zero in its spectrum. The matrix $L(G)=D(G)-A(G)$ is called the Laplacian matrix of $G$. The second smallest eigenvalue of the Laplacian of a graph $G$, best-known as
the algebraic connectivity of $G$, denoted by $a(G)$.

\begin{lemma}{\bf (\cite{M})}\label{le2,1} 
If $a$, $b$ are real numbers, where $a<b$, and $n$ is an integer, let $\mathcal{S}_n^{a, b}$ be the set of all symmetric matrices whose entries are between $a$ and $b$. Then for every integer $k$, $2\leq k \leq n$, and every $M\in \mathcal{S}_n^{a, b}$ we have
$$S_k(M)\leq \frac{(b-a)n}{2}(1+\sqrt{k})+\max\{0, a\}.$$
\end{lemma}

\begin{lemma}{\bf (\cite{RSR})}\label{le2,2} 
Let $M$ be an $n\times n$ matrix with nonnegative eigenvalues. Let $1\leq k \leq n-1$. Then
$$S_k(M)\leq \frac{k(tr(M))}{n}+\sqrt{\frac{k(n-k)}{n}f(M)},$$
where
$$f(M)=\sum\limits_{i=1}^{n}\sum\limits_{k=1}^{n}m_{ik}m_{ki}-\frac{(tr(M))^2}{n}.$$
\end{lemma}

\begin{lemma}{\bf (\cite{BDFG, FT})}\label{le2,3} 
Let $G$ be a graph with $n$ vertices and $m$ edges. Then
$$\frac{4m^2}{n}+\frac{1}{2}(\Delta-\delta)^2\leq Z_1(G)\leq \frac{4m^2}{n}+\frac{n}{4}(\Delta-\delta)^2.$$
\end{lemma}

\begin{lemma}{\bf (\cite{N})}\label{le2,4} 
Let $G$ be a graph with $n$ vertices. Then
$$\sqrt{\frac{Z_1}{n}}\leq \lambda_1(A_{\alpha}(G))\leq \Delta.$$
\end{lemma}

\begin{lemma}{\bf (\cite{CLM})}\label{le2,5} 
Let $G$ be a graph of with $n$ vertices and $\Delta(G)<n-1$. If $\frac{1}{2}<\alpha<1$, then
$$\lambda_2(A_{\alpha}(G))\leq \alpha(n-2)$$
If the equality holds, then the complement of $G$ has at least one component isomorphic to $K_2$.
\end{lemma}

\begin{lemma}{\bf (\cite{HMT})}\label{le2,6} 
Let $T$ be a tree with $n$ vertices. Then $S_{k}(L(T))\leq n+2k-2$ for $1\leq k\leq n$.
\end{lemma}

\begin{lemma}{\bf (\cite{B})}\label{le2,7} 
Let $M$ be an $n \times n$ Hermitian matrix. Then for $1\leq k \leq n$,
$$\sum\limits_{i=1}^{k}\lambda_{i}(M)=\max\sum\limits_{i=1}^{k}\langle MX_i, X_i\rangle,$$
where the maximum is taken over all orthonormal $k$-tuples of vectors $\{X_1, \ldots, X_k\}$ in $\mathbb{C}^{n}$.
\end{lemma}

\begin{lemma}{\bf (\cite{LXS})}\label{le2,8} 
Let $G$ be a graph with $n$ vertices. If $e\in E(G)$ and $\alpha\geq \frac{1}{2}$, then
$$\lambda_i(A_{\alpha}(G))\geq \lambda_i(A_{\alpha}(G-e))$$
for $1\leq i \leq n$.
\end{lemma}

\begin{lemma}{\bf (\cite{LDS})}\label{le2,9} 
Let $G$ be a graph with $n$ vertices and degree sequence $d_1\geq d_2\geq \cdots \geq d_n$. Then
$$\lambda_k(A_{\alpha}(G))\leq \alpha d_k+(1-\alpha)(n-k)\eqno{(2.1)}$$
If equality in $(2.1)$ holds and $0<\alpha<1$, then $G$ has an induced subgraph $H\cong K_{n-k+1}$ such that $d(v_i)=\delta$ for all $v_i\in V(H)$.
\end{lemma}

\begin{lemma}{\bf (\cite{CS})}\label{le2,10} 
Let $G$ be a graph with $n$ vertices and $m\geq 1$ edges. Then
$\lambda_i(Q(G))=\lambda_i(A(\mathcal{L}(G)))+2$, $i=1,2,\ldots,s$, where $s=\min\{n,m\}$.
Further if $m>n$, we have $\lambda_i(A(\mathcal{L}(G)))=-2$ for $i\geq n+1$ and if $n>m$, we have $\lambda_i(Q(G))=0$ for $i\geq m+1$.
\end{lemma}

\begin{lemma}{\bf (\cite{CLWM})}\label{le2,11} 
For any $K_3$-free and $C_4$-free graph $G$, $A(G^2)=A^2(G)-L(G)$.
\end{lemma}

\begin{lemma}\label{le2,12} 
If $1\leq k \leq n$, then $\sum\limits_{i=1}^{k}\cos \frac{i\pi}{n}=\frac{1}{2}\csc\frac{\pi}{2n} \sin\frac{(2k+1)\pi}{2n}-\frac{1}{2}$.
\end{lemma}

\begin{proof} For $1\leq k \leq n$, we have
\begin{eqnarray*}
\sum\limits_{i=1}^{k}\cos \frac{i\pi}{n} & = & \frac{\left(1+\cos \frac{\pi}{n}\right)\sin\frac{k\pi}{n}}{2\sin\frac{\pi}{n}}+\frac{1}{2}\cos\frac{k\pi}{n}-\frac{1}{2}\\
& = & \frac{1}{2}\cot\frac{\pi}{2n} \sin\frac{k\pi}{n} +\frac{1}{2}\cos\frac{k\pi}{n}-\frac{1}{2}\\
& = & \frac{1}{2}\csc\frac{\pi}{2n} \sin\frac{(2k+1)\pi}{2n}-\frac{1}{2}.
\end{eqnarray*}
The proof is completed. $\Box$
\end{proof}

\begin{lemma}\label{le2,13} 
If $0\leq \beta<\alpha \leq 1$ and $G$ is a graph with $n$ vertices, then
$$S_{k}(A_{\beta}(G))\leq S_{k}(A_{\alpha}(G))$$
for $1\leq k \leq n$. If $G$ is connected, then inequality is strict, unless $k=1$ and $G$ is regular.
\end{lemma}

\begin{proof} If $0\leq \beta<\alpha \leq 1$, from Proposition 4 in \cite{N}, then $\lambda_{k}(A_{\beta}(G))\leq \lambda_{k}(A_{\alpha}(G))$ for $1\leq k \leq n$. Thus $S_{k}(A_{\beta}(G))\leq S_{k}(A_{\alpha}(G))$, and the proof follows. $\Box$
\end{proof}

\section{\large Upper bounds on the sum of the largest $A_{\alpha}$-eigenvalues in terms of vertex degrees}

Nikiforov \cite{N} showed that $A_{\alpha}(G)$ is a positive semi-definite matrix for $\frac{1}{2}\leq \alpha <1$. Further, $G$ has no isolated vertices, then $A_{\alpha}(G)$ is positive definite. Let $\alpha_0(G)$ be the smallest $\alpha$ such that $A_{\alpha}(G)$ is positive semidefinite for $\alpha_0(G)\leq \alpha \leq 1$. Nikiforov and Rojo \cite{NR} found $\alpha_0(G)$ if $G$ is regular or $G$ contains a bipartite component and given a lower bound on $\alpha_0(G)$ of $\chi$-colorable graphs.

\begin{theorem}\label{th3,1} 
Let $G\neq K_n$ be a graph with $n$ vertices and maximum degree $\Delta$.

{\normalfont (i)} If $0\leq \alpha< \frac{1}{\Delta+1}$, then
$S_k(A_{\alpha}(G))\leq \frac{(1-\alpha)n}{2}(1+\sqrt{k})$
for $2\leq k \leq n$.

{\normalfont (ii)} If $\frac{1}{\Delta+1}\leq \alpha< 1$, then
$S_k(A_{\alpha}(G))\leq \frac{\alpha \Delta n}{2}(1+\sqrt{k})$
for $2\leq k \leq n$.
\end{theorem}

\begin{proof} In this proof we use Lemma \ref{le2,1} with $a=0$ and $b=1-\alpha$ for $0\leq \alpha< \frac{1}{\Delta+1}$. Then
$$S_k(A_{\alpha}(G))\leq \frac{(1-\alpha)n}{2}(1+\sqrt{k}).$$
By a similar reasoning as above, the second part of the theorem follows. $\Box$
\end{proof}

\begin{theorem}\label{th3,2} 
Let $\frac{1}{2}\leq \alpha <1$ and $G$ be a graph with $n$ vertices and $m$ edges. If $1\leq k\leq n-1$, then
$$S_{k}(A_{\alpha}(G))\leq \frac{2\alpha km}{n}+\sqrt{\frac{k(n-k)}{n}\left(\alpha^2Z_1+2m(1-\alpha)^2-\frac{4\alpha^2m^2}{n}\right)}.$$
\end{theorem}

\begin{proof} Since $tr(A_{\alpha}(G))=2\alpha m$, $\sum\limits_{i=1}^{n}\sum\limits_{k=1}^{n}a_{ik}a_{ki}=\alpha^2Z_1+2m(1-\alpha)^2$ and $A_{\alpha}(G)$ is a positive semi-definite matrix for $\frac{1}{2}\leq \alpha <1$, by Lemma \ref{le2,2}, we have the proof. $\Box$

\end{proof}

The following result is direct corollary of Lemma \ref{le2,3} and Theorem \ref{th3,2}.

\begin{corollary}\label{cor3,1} 
Let $\frac{1}{2}\leq \alpha <1$ and $G$ be a graph with $n$ vertices and $m$ edges. If $1\leq k\leq n-1$, then
$$S_k(A_{\alpha}(G))\leq \frac{2\alpha km}{n}+\sqrt{\frac{k(n-k)}{n}\left(2m(1-\alpha)^2+\frac{\alpha^2n}{4}(\Delta-\delta)^2\right)}.$$
\end{corollary}

\begin{theorem}\label{th3,3} 
Let $\frac{1}{2}< \alpha <1$ and $G$ be a graph with $n$ vertices. If $1\leq k\leq n-1$ and $G$ has no isolated vertices, then
$$S_{k}(A_{\alpha}(G))\leq 2\alpha m-(n-k)\left(\frac{\det(A_{\alpha}(G))}{\lambda_1(A_{\alpha}(G)\lambda_2^{k-1}(A_{\alpha}(G))}\right)^{\frac{1}{n-k}}\eqno{(3.1)}$$
with equality if and only if $\lambda_2(A_{\alpha}(G)=\cdots=\lambda_k(A_{\alpha}(G)$ and $\lambda_{k+1}(A_{\alpha}(G)=\cdots=\lambda_n(A_{\alpha}(G)$.
\end{theorem}

\begin{proof} Since $\frac{1}{2}< \alpha <1$ and $G$ has no isolated vertices, we know that $A_{\alpha}(G)$ is positive definite. By the arithmetic-geometric mean inequality, we have
\begin{eqnarray*}
S_{k}(A_{\alpha}(G)) & = & \lambda_1(A_{\alpha}(G))+\lambda_2(A_{\alpha}(G))+\cdots+\lambda_k(A_{\alpha}(G))\\
& = & 2\alpha m-(\lambda_{k+1}(A_{\alpha}(G))+\lambda_{k+2}(A_{\alpha}(G))+\cdots+\lambda_n(A_{\alpha}(G)))\\
& \leq & 2\alpha m-(n-k)\left(\prod\limits_{i=k+1}^{n}\lambda_i(A_{\alpha}(G))\right)^{\frac{1}{n-k}}\\
& = & 2\alpha m-(n-k)\left(\frac{\det(A_{\alpha}(G))}{\prod\limits_{i=1}^{k}\lambda_i(A_{\alpha}(G))}\right)^{\frac{1}{n-k}}\\
& \leq & 2\alpha m-(n-k)\left(\frac{\det(A_{\alpha}(G))}{\lambda_1(A_{\alpha}(G)\lambda_2^{k-1}(A_{\alpha}(G))}\right)^{\frac{1}{n-k}}
\end{eqnarray*}
with equality if and only if $\lambda_2(A_{\alpha}(G)=\cdots=\lambda_k(A_{\alpha}(G)$ and $\lambda_{k+1}(A_{\alpha}(G)=\cdots=\lambda_n(A_{\alpha}(G)$. This completes the proof. $\Box$
\end{proof}

If $G$ is a complete graph $K_n$, then equality holds in (3.1). However, there are many other cases of equality some of which are rather complicated
and their complete description seems difficult.

\begin{problem}\label{pro3,1} 
Characterize the graphs for which equality holds in $(3.1)$.
\end{problem}

\begin{corollary}\label{cor3,2} 
Let $\frac{1}{2}< \alpha <1$ and $G$ be a graph with $n$ vertices and maximum degree $\Delta<n-1$. If $1\leq k\leq n-1$ and $G$ has no isolated vertices, then
$$S_{k}(A_{\alpha}(G))\leq 2\alpha m-(n-k)\left(\frac{\det(A_{\alpha}(G))}{\alpha^{k-1}\Delta(n-2)^{k-1}}\right)^{\frac{1}{n-k}}.$$
\end{corollary}

\begin{proof} By Lemma \ref{le2,4}, we have $\lambda_1(A_{\alpha}(G))\leq \Delta$. By Lemma \ref{le2,5} and Theorem \ref{th3,3}, we have the proof. $\Box$
\end{proof}

\begin{corollary}\label{cor3,3} 
Let $G$ be a connected non-bipartite graph. Then
$$S_{k}(Q(G))\leq 2 m-(n-k)\left(\frac{\det(Q(G))}{\lambda_1(Q(G)\lambda_2^{k-1}(Q(G))}\right)^{\frac{1}{n-k}}.$$
\end{corollary}

\begin{theorem}\label{th3,4} 
Let $0\leq \alpha <\alpha_0(G)$ and $G$ be a graph with $n$ vertices and $m$ edges, and let $p$ be the positive inertia index of $A_{\alpha}(G)$.
Then
$$S_{p}(A_{\alpha}(G))\leq 2\alpha m +\frac{1}{2}(2m(1-\alpha)^2+\alpha^2Z_1)\sqrt{\frac{n(n-p)}{Z_1}}.$$
\end{theorem}

\begin{proof} By Lemma \ref{le2,4}, we have $\lambda_1(A_{\alpha}(G))\geq \sqrt{\frac{Z_1}{n}}$. We assume that $$\sum\limits_{i=1}^{n-p}\lambda_{n-i+1}^2(A_{\alpha}(G))>\frac{n(2m(1-\alpha)^2+\alpha^2Z_1)^2}{4Z_1},$$
in which case
\begin{eqnarray*}
2m(1-\alpha)^2+\alpha^2Z_1 & = & \sum\limits_{i=1}^{p}\lambda_{i}^2(A_{\alpha}(G))+\sum\limits_{i=1}^{n-p}\lambda_{n-i+1}^2(A_{\alpha}(G))\\
& \geq & \lambda_1^2(A_{\alpha}(G))+\sum\limits_{i=1}^{n-p}\lambda_{n-i+1}^2(A_{\alpha}(G))\\
& > & \frac{Z_1}{n}+\frac{n(2m(1-\alpha)^2+\alpha^2Z_1)^2}{4Z_1}.
\end{eqnarray*}
This implies that
$$\left(\sqrt{\frac{Z_1}{n}}-\frac{1}{2}(2m(1-\alpha)^2+\alpha^2Z_1)\sqrt{\frac{n}{Z_1}}\right)^2<0,$$
which is a contradiction. Thus
$$\sum\limits_{i=1}^{n-p}\lambda_{n-i+1}^2(A_{\alpha}(G))\leq \frac{n(2m(1-\alpha)^2+\alpha^2Z_1)^2}{4Z_1}.$$
By the Cauchy-Schwarz inequality, we have
\begin{eqnarray*}
S_{p}(A_{\alpha}(G))& = & 2\alpha m-\sum\limits_{i=1}^{n-p}\lambda_{n-i+1}(A_{\alpha}(G))\\
& \leq &  2\alpha m+\sqrt{(n-p)\sum\limits_{i=1}^{n-p}\lambda_{n-i+1}^2(A_{\alpha}(G))}\\
& \leq & 2\alpha m +\frac{1}{2}(2m(1-\alpha)^2+\alpha^2Z_1)\sqrt{\frac{n(n-p)}{Z_1}}.
\end{eqnarray*}
This completes the proof. $\Box$
\end{proof}

By Lemma \ref{le2,3} and Theorem \ref{th3,4}, we obtain the following corollary.

\begin{corollary}\label{cor3,4} 
Let $0\leq \alpha <\alpha_0(G)$ and $G$ be a graph with $n$ vertices and $m$ edges, and let $p$ be the positive inertia index of $A_{\alpha}(G)$.
Then
$$S_{p}(A_{\alpha}(G))\leq 2\alpha m +\left(mn(1-\alpha)^2+2\alpha^2m^2+\frac{\alpha^2n^2}{8}(\Delta-\delta)^2\right)\sqrt{\frac{2(n-p)}{8m^2+n(\Delta-\delta)^2}}.$$
\end{corollary}

\section{\large On the sum of the $k$ largest $A_{\alpha}$-eigenvalues of a tree}

\begin{theorem}\label{th4,1} 
 Let $G$ be a bipartite graph with $n$ vertices and $m$ edges, and let $\eta$ be the nullity of $G$. Then
 $$S_{k}(A(G))\leq
\begin{dcases}
\sqrt{km}, & \text{if}\,\, 1\leq k \leq \left\lfloor\frac{n-\eta}{2}\right\rfloor;\\
\sqrt{\left\lfloor\frac{n-\eta}{2}\right\rfloor m}, & \text{if}\,\,\left\lfloor\frac{n-\eta}{2}\right\rfloor< k \leq \left\lfloor\frac{n+\eta}{2}\right\rfloor;\\
\sqrt{(n-k)m}, & \text{if}\,\, \left\lfloor\frac{n+\eta}{2}\right\rfloor< k \leq n.
\end{dcases}
$$
\end{theorem}

\begin{proof} Since $G$ is a bipartite graph, we know that eigenvalues of $A(G)$ are symmetric with respect to the origin, that is $S_k(A(G))=S_{n-k}(A(G))$ for $\left\lfloor\frac{n+\eta}{2}\right\rfloor< k\leq n-1$. Since $\sum\limits_{i=1}^{n}\lambda_i^2(A(G))=2m$, we have $\sum_{i=1}^{\left\lfloor\frac{n-\eta}{2}\right\rfloor}\lambda_i^2(A(G))=m$. By the Cauchy-Schwarz inequality, we have
$$S_{k}(A(G))=\sum\limits_{i=1}^{k}\lambda_i(A(G))\leq \sqrt{k\sum\limits_{i=1}^{k}\lambda_i^2(A(G))}\leq \sqrt{km}$$
for $1\leq k \leq \left\lfloor\frac{n-\eta}{2}\right\rfloor$. This completes the proof. $\Box$
\end{proof}

If $T$ is a tree with $n$ vertices and matching number $\beta$, Cvetkovi\'{c} and Gutman \cite{CG} showed that $\eta=n-2\beta$. Thus we have

\begin{corollary}\label{cor4,1} 
 Let $T$ be a tree with $n$ vertices and matching number $\beta$. Then
 $$S_{k}(A(T))\leq
\begin{dcases}
\sqrt{k(n-1)}, & \text{if}\,\, 1\leq k \leq \beta;\\
\sqrt{\beta (n-1)}, & \text{if}\,\,\beta< k \leq n-\beta;\\
\sqrt{(n-k)(n-1)}, & \text{if}\,\, n-\beta< k \leq n.
\end{dcases}
$$
\end{corollary}

\begin{theorem}\label{th4,2} 
Let $T$ be a tree with $n$ vertices.

{\normalfont (i)} If $0\leq \alpha< \frac{1}{2}$, then
$$S_k(A_{\alpha}(T))\leq
\begin{dcases}
\alpha(n+2k-2)+(1-2\alpha)\sqrt{k(n-1)}, & \text{if}\,\, 1\leq k \leq \beta;\\
\alpha(n+2k-2)+(1-2\alpha)\sqrt{\beta (n-1)}, & \text{if}\,\,\beta< k \leq n-\beta;\\
\alpha(n+2k-2)+(1-2\alpha)\sqrt{(n-k)(n-1)}, & \text{if}\,\, n-\beta< k \leq n.
\end{dcases}
$$

{\normalfont (ii)} If $\frac{1}{2} \leq \alpha<1$, then
$$S_k(A_{\alpha}(T))\leq \alpha(n+2k-2)$$
for $1\leq k \leq n$.
\end{theorem}

\begin{proof} {\normalfont (i)} From Proposition 2.5 in \cite{CRS}, we know that $Q(G)$ and $L(G)$ share the same eigenvalues if and only if $G$ is bipartite. By Lemma \ref{le2,6}, we have
$S_k(Q(T))\leq n+2k-2$
for $1\leq k \leq n$. Since $A_{\alpha}(T)=\alpha Q(T)+(1-2\alpha)A(T)$ for $0\leq \alpha< \frac{1}{2}$, by Theorem \ref{th1,1} and Corollary \ref{cor4,1}, we have
\begin{eqnarray*}
S_k(A_{\alpha}(T)) & \leq & \alpha S_k(Q(T))+(1-2\alpha)S_k(A(T))\\
& \leq & \begin{dcases}
\alpha(n+2k-2)+(1-2\alpha)\sqrt{k(n-1)}, & \text{if}\,\, 1\leq k \leq \beta;\\
\alpha(n+2k-2)+(1-2\alpha)\sqrt{\beta (n-1)}, & \text{if}\,\,\beta< k \leq n-\beta;\\
\alpha(n+2k-2)+(1-2\alpha)\sqrt{(n-k)(n-1)}, & \text{if}\,\, n-\beta< k \leq n.
\end{dcases}
\end{eqnarray*}

{\normalfont (ii)} Since $Q(G)$ is a real symmetric matrix, the spectrum of $Q(G)$ majorizes its main diagonal, that is, $S_k(Q(G))\geq S_k(D(G))$. Since $A_{\alpha}(T)=(1-\alpha) Q(T)+(2\alpha-1)D(T)$ for $\frac{1}{2}\leq \alpha< 1$, by Theorem \ref{th1,1} and Lemma \ref{le2,6}, we have
\begin{eqnarray*}
S_k(A_{\alpha}(T)) & \leq & (1-\alpha) S_k(Q(T))+(2\alpha-1)S_k(D(T))\\
& \leq & (1-\alpha) S_k(Q(T))+(2\alpha-1)S_k(Q(T))\\
& = & \alpha S_k(Q(T))\\
& \leq & \alpha(n+2k-2).
\end{eqnarray*}
The proof is completed. $\Box$
\end{proof}

\begin{theorem}\label{th4,3} 
Let $G$ be a connected graph with $n$ vertices and $m$ edges, and let $\beta'$ be the matching number of the spanning tree of $G$.

{\normalfont (i)} If $0\leq \alpha< \frac{1}{2}$, then
 $$S_k(A_{\alpha}(G))\leq
\begin{dcases}
\alpha(n+2k-2)+(1-2\alpha)\sqrt{k(n-1)}+m-n+1, & \text{if}\,\, 1\leq k \leq \beta';\\
\alpha(n+2k-2)+(1-2\alpha)\sqrt{\beta' (n-1)}+m-n+1, & \text{if}\,\,\beta'< k \leq n-\beta';\\
\alpha(n+2k-2)+(1-2\alpha)\sqrt{(n-k)(n-1)}+m-n+1, & \text{if}\,\, n-\beta'< k \leq n.
\end{dcases}
$$

{\normalfont (ii)} If $\frac{1}{2}\leq \alpha< 1$, then
$$S_k(A_{\alpha}(G))\leq \alpha(2k+2m-n)$$
for $2\leq k \leq n$.
\end{theorem}

\begin{proof} {\normalfont (i)} Let $T$ be a spanning tree of $G$. If $0\leq \alpha< \frac{1}{2}$, by Theorems \ref{th1,1} and \ref{th4,2}, we have
 \begin{eqnarray*}
S_k(A_{\alpha}(G))& \leq & S_k(A_{\alpha}(T))+(m-n+1)S_k(A_{\alpha}(K_2\cup (n-2)K_1))\\
& = & S_k(A_{\alpha}(T))+m-n+1\\
& \leq & \begin{dcases}
\alpha(n+2k-2)+(1-2\alpha)\sqrt{k(n-1)}+m-n+1, & \text{if}\,\, 1\leq k \leq \beta';\\
\alpha(n+2k-2)+(1-2\alpha)\sqrt{\beta' (n-1)}+m-n+1, & \text{if}\,\,\beta'< k \leq n-\beta';\\
\alpha(n+2k-2)+(1-2\alpha)\sqrt{(n-k)(n-1)}+m-n+1, & \text{if}\,\, n-\beta'< k \leq n.
\end{dcases}
\end{eqnarray*}

{\normalfont (ii)} If $\frac{1}{2}\leq \alpha< 1$, by Theorems \ref{th1,1} and \ref{th4,2}, we have
\begin{eqnarray*}
S_k(A_{\alpha}(G))& \leq & S_k(A_{\alpha}(T))+(m-n+1)S_k(A_{\alpha}(K_2\cup (n-2)K_1))\\
& = & S_k(A_{\alpha}(T))+2\alpha(m-n+1)\\
& \leq & \alpha(n+2k-2)+2\alpha (m-n+1)\\
& = & \alpha(2k+2m-n)
\end{eqnarray*}
for $2\leq k\leq n$.

This completes the proof. $\Box$
\end{proof}

\begin{theorem}\label{th4,4} 
Let $P_n$ be a path with $n$ vertices.

{\normalfont (i)}If $0\leq \alpha < \frac{1}{2}$, then
\begin{eqnarray*}
S_k(A_{\alpha}(P_n)) & \leq &
2\alpha k+\alpha-1+\alpha \csc\frac{\pi}{2n} \sin\frac{(2k+1)\pi}{2n} \\
& & +(1-2\alpha)\csc\frac{\pi}{2(n+1)} \sin\frac{(2k+1)\pi}{2(n+1)}
\end{eqnarray*}
for $1\leq  k \leq n$.

{\normalfont (ii)}If $\frac{1}{2}\leq \alpha \leq 1$, then
$$S_k(A_{\alpha}(P_n))\leq 2\alpha k+(1-\alpha)\left(\csc\frac{\pi}{2n} \sin\frac{(2k+1)\pi}{2n}-1\right)$$
for $1\leq  k \leq n$.
\end{theorem}

\begin{proof}{\normalfont (i)} Since $A_{\alpha}(P_n)=\alpha Q(P_n)+(1-2\alpha)A(P_n)$ for $0\leq \alpha< \frac{1}{2}$, by Theorem \ref{th1,1} and Lemma \ref{le2,12}, we have
\begin{eqnarray*}
S_k(A_{\alpha}(P_n)) & \leq & \alpha S_k(Q(P_n))+(1-2\alpha)S_k(A(P_n))\\
& = &  2\alpha\sum\limits_{i=1}^{k}\left(1+\cos\frac{i\pi}{n}\right)+2(1-2\alpha)\sum\limits_{i=1}^{k}\cos \frac{i\pi}{n+1}\\
& = & 2 \alpha k+\alpha\left(\csc\frac{\pi}{2n} \sin\frac{(2k+1)\pi}{2n}-1\right) \\
& & +(1-2\alpha)\left(\csc\frac{\pi}{2(n+1)} \sin\frac{(2k+1)\pi}{2(n+1)}-1\right)\\
& = & 2\alpha k+\alpha-1+\alpha \csc\frac{\pi}{2n} \sin\frac{(2k+1)\pi}{2n} \\
& & +(1-2\alpha)\csc\frac{\pi}{2(n+1)} \sin\frac{(2k+1)\pi}{2(n+1)}
\end{eqnarray*}
for $1\leq k \leq n$.

{\normalfont (ii)} Since $A_{\alpha}(P_n)=(1-\alpha) Q(P_n)+(2\alpha-1)D(P_n)$ for $\frac{1}{2}\leq \alpha< 1$, by Theorem \ref{th1,1} and Lemma  \ref{le2,12}, we have
\begin{eqnarray*}
S_k(A_{\alpha}(P_n)) & \leq & (1-\alpha)S_k(Q(P_n))+(2\alpha-1)S_k(D(P_n))\\
& = &  2(1-\alpha)\sum\limits_{i=1}^{k}\left(1+\cos\frac{i\pi}{n}\right)+2(2\alpha-1)k\\
& = & 2k(1-\alpha)+(1-\alpha)\left(\csc\frac{\pi}{2n} \sin\frac{(2k+1)\pi}{2n}-1\right) \\
& & +2(2\alpha-1)k\\
& = & 2\alpha k+(1-\alpha)\left(\csc\frac{\pi}{2n} \sin\frac{(2k+1)\pi}{2n}-1\right)
\end{eqnarray*}
for $1\leq k \leq n$. The proof is completed. $\Box$
\end{proof}

\begin{corollary}\label{cor4,2} 
Let $P_n$ be a path with $n$ vertices. If $0\leq \alpha <1$, then $S_k(A_{\alpha}(P_n)) < 2k$ for $1\leq k \leq n$.
\end{corollary}

\section{\large Lower bounds on the sum of the largest $A_{\alpha}$-eigenvalues}

\begin{theorem}\label{th5,1} 
 Let $G$ be a graph with maximum degree $\Delta$.

{\normalfont (i)}If $0\leq \alpha \leq \frac{1}{2}$, then
$$S_k(A_{\alpha}(G))\geq (1-\alpha) S_k(Q(G))+(2\alpha-1)k\Delta.$$

{\normalfont (ii)}If $\frac{1}{2}\leq \alpha \leq 1$, then
$$S_k(A_{\alpha}(G))\geq \alpha S_k(Q(G))+(1-2\alpha)S_k(A(G)).$$

If $G$ is a regular graph, then the equality in the above inequalities must hold.

\end{theorem}

\begin{proof}{\normalfont (i)} If $0\leq \alpha \leq \frac{1}{2}$, then $\frac{1}{2}\leq 1-\alpha \leq 1$. It follows that $A_{1-\alpha}(G)=\alpha Q(G)+(1-2\alpha)D(G)$. Since $A_{\alpha}(G)+A_{1-\alpha}(G)=Q(G)$, by Theorem \ref{th1,1}, we have
\begin{eqnarray*}
S_{k}(A_{\alpha}(G)) & \geq &  S_{k}(Q(G))-S_{k}(A_{1-\alpha}(G))\\
& \geq & S_{k}(Q(G))-\alpha S_{k}(Q(G))-(1-2\alpha)S_{k}(D(G))\\
& \geq & (1-\alpha) S_k(Q(G))+(2\alpha-1)k\Delta.
\end{eqnarray*}

{\normalfont (ii)} If $\frac{1}{2}\leq \alpha \leq 1$, then $0 \leq 1-\alpha \leq \frac{1}{2}$. It follows that $A_{1-\alpha}(G)=(1-\alpha) Q(G)+(2\alpha-1)A(G)$. Since $A_{\alpha}(G)+A_{1-\alpha}(G)=Q(G)$, by Theorem \ref{th1,1}, we have
\begin{eqnarray*}
S_{k}(A_{\alpha}(G)) & \geq &  S_{k}(Q(G))-S_{k}(A_{1-\alpha}(G))\\
& \geq & S_{k}(Q(G))-(1-\alpha) S_{k}(Q(G))-(2\alpha-1)S_{k}(A(G))\\
& \geq & \alpha S_k(Q(G))+(1-2\alpha)S_k(A(G)).
\end{eqnarray*}
This completes the proof. $\Box$
\end{proof}

\begin{corollary}\label{cor5,1} 
Let $P_n$ be a path with $n$ vertices.

{\normalfont (i)}If $0\leq \alpha \leq \frac{1}{2}$, then
$$S_k(A_{\alpha}(P_n)) \geq
2\alpha k+(1-\alpha)\left(\csc\frac{\pi}{2n} \sin\frac{(2k+1)\pi}{2n}-1\right)$$
for $1\leq  k \leq n$.

{\normalfont (ii)}If $\frac{1}{2}\leq \alpha \leq 1$, then
\begin{eqnarray*}
S_k(A_{\alpha}(P_n)) & \geq &
2\alpha k+\alpha-1+\alpha \csc\frac{\pi}{2n} \sin\frac{(2k+1)\pi}{2n} \\
& & +(1-2\alpha)\csc\frac{\pi}{2(n+1)} \sin\frac{(2k+1)\pi}{2(n+1)}
\end{eqnarray*}
for $1\leq  k \leq n$.
\end{corollary}

\begin{theorem}\label{th5,2} 
Let $G$ be a $r$-regular graph.

{\normalfont (i)} Let $t$ be the number of vertex-disjoint cliques in $G$. If $0\leq \alpha \leq 1$, then
$$S_k(A_{\alpha}(G))\geq \alpha kr+(1-\alpha)(r-k+1)$$
for $1\leq k \leq t+1$.

{\normalfont (ii)} Let $g_1\geq g_2\geq \cdots \geq g_c$ and $C_{g_1}$, $C_{g_2}$, \ldots, $C_{g_{c}}$ be the vertex-disjoint induced cycles of length even in $G$. If $0\leq \alpha \leq 1$, then
$$S_k(A_{\alpha}(G))\geq (\alpha k+1-\alpha)r+2(1-\alpha)\sum\limits_{i=1}^{k-1}(1-\frac{4}{g_{i}})$$
for $1\leq k \leq c+1$.
\end{theorem}

\begin{proof} {\normalfont (i)} Let $X_1$ be the vector with all entries equal to $1$. Then $X_1$ is an eigenvector corresponding to $\lambda_{1}(A(G))$. Let $K_{\omega_1}$, $K_{\omega_2}$, \ldots, $K_{\omega_t}$ be the vertex-disjoint cliques in $G$. Then we take
\begin{eqnarray*}
X_2 & = & (\underbrace{x_{1(2)}, x_{2(2)}, \ldots, x_{\omega_{1(2)}}}_{\omega_1}, 0, \ldots, 0),\\
X_3 & = & (\underbrace{0, \ldots, 0}_{\omega_1}, \underbrace{x_{1(3)}, x_{2(3)}, \ldots, x_{\omega_{2(3)}}}_{\omega_2}, 0, \ldots, 0),\cdots,\\
X_{t+1} & = & (\underbrace{0, \ldots, 0}_{\omega_1}, \underbrace{0, \ldots, 0}_{\omega_2}, \ldots, \underbrace{x_{1(t+1)}, x_{2(t+1)}, \ldots, x_{ \omega_{t(t+1)}}}_{\omega_{t}}, 0, \ldots, 0)
\end{eqnarray*}
satisfying
$$
\begin{dcases}
x_{1(s)}+ x_{2(s)}+ \cdots+ x_{\omega_{s-1(s)}} =  0\\
x_{1(s)}^2+ x_{2(s)}^2+ \cdots+ x_{\omega_{s-1(s)}}^2  =  1
\end{dcases}
$$
for $s=2, 3, \ldots, t+1$. Thus we have $2\sum\limits_{i<j}x_{i(s)}x_{j(s)}=-1$. Since $G$ is a $r$-regular graph and the vectors $X_1, X_2, \ldots, X_{t+1}$ are orthogonal, by Lemma \ref{le2,7}, we have
\begin{eqnarray*}
S_k(A_{\alpha}(G)) & = & \alpha kr +(1-\alpha)S_{k}(A(G))\\
& = & \alpha kr +(1-\alpha)\max\sum\limits_{i=1}^{k}\langle A(G)X_i, X_i\rangle\\
& \geq & \alpha kr +(1-\alpha)\sum\limits_{i=1}^{k}\langle A(G)X_i, X_i\rangle\\
& = & \alpha kr +(1-\alpha)(r+2(k-1)\sum\limits_{i<j}x_{i(k-1)}x_{j(k-1)})\\
& = & \alpha kr +(1-\alpha)(r-k+1).
\end{eqnarray*}
for $1\leq k \leq t+1$.

{\normalfont (ii)} Let $X_1$ be the vector with all entries equal to $1$. Then $X_1$ is an eigenvector corresponding to $\lambda_{1}(A(G))$. Let $g_1\geq g_2\geq \cdots \geq g_c$ and $C_{g_1}$, $C_{g_2}$, \ldots, $C_{g_{c}}$ be the vertex-disjoint induced cycles of length even in $G$. Then we take
\begin{eqnarray*}
X_2 & = & \frac{1}{\sqrt{g_1}}(\underbrace{1, \ldots, 1}_{\frac{g_1}{2}},\underbrace{-1, \ldots, -1}_{\frac{g_1}{2}}, \ldots, 0),\\
X_3 & = & \frac{1}{\sqrt{g_2}} (\underbrace{0, \ldots, 0}_{g_1}, \underbrace{1,\ldots,1}_{\frac{g_2}{2}},\underbrace{-1, \ldots, -1}_{\frac{g_2}{2}}, 0, \ldots, 0),\cdots,
\end{eqnarray*}
\begin{eqnarray*}
X_{c+1} & = & \frac{1}{\sqrt{g_c}} (\underbrace{0, \ldots, 0}_{g_1}, \underbrace{0, \ldots, 0}_{g_2}, \ldots, \underbrace{1, \ldots, 1}_{\frac{g_{c}}{2}}, \underbrace{-1, \ldots, -1}_{\frac{g_{c}}{2}}, 0, \ldots, 0).
\end{eqnarray*}
Thus we have $2\sum\limits_{v_iv_j\in E(C_{g_s})}x_{i}x_{j}=2(1-\frac{4}{g_{s}})$ for $s=1, 2, \ldots, c $. Since $G$ is a $r$-regular graph and the vectors $X_1, X_2, \ldots, X_{c+1}$ are orthogonal, by Lemma \ref{le2,7}, we have
\begin{eqnarray*}
S_k(A_{\alpha}(G)) & = & \alpha kr +(1-\alpha)S_{k}(A(G))\\
& = & \alpha kr +(1-\alpha)\max\sum\limits_{i=1}^{k}\langle A(G)X_i, X_i\rangle\\
& \geq & \alpha kr +(1-\alpha)\sum\limits_{i=1}^{k}\langle A(G)X_i, X_i\rangle\\
& = & \alpha kr +(1-\alpha)(r+2\sum\limits_{s=1}^{k-1}\sum\limits_{v_iv_j\in E(C_{g_{s}})}x_{i}x_{j})\\
& = & \alpha kr +(1-\alpha)\left(r+2\sum\limits_{i=1}^{k-1}(1-\frac{4}{g_{i}})\right)\\
& = & (\alpha k+1-\alpha)r+2(1-\alpha)\sum\limits_{i=1}^{k-1}(1-\frac{4}{g_{i}})
\end{eqnarray*}
for $1\leq k \leq c+1$.

This completes the proof. $\Box$

\end{proof}

\begin{theorem}\label{th5,3} 
Let $G$ be a connected bipartite graph with bipartition $V(G)=X \cup Y$, $|X|=s$ and $|Y|=t$. Let $m$ and $\beta$ be the number of edges and matching number of $G$, respectively. If $0\leq \alpha \leq 1$, then
$$S_k(A_{\alpha}(G))\geq \frac{\alpha m}{2}\left(\frac{1}{s}+\frac{1}{t}\right)+\frac{(1-\alpha)m}{\sqrt{st}}+(k-1)\left(\alpha -\frac{2(1-\alpha)\sqrt{st}}{s+t}\right)$$
for $1\leq k \leq \beta+1$.
\end{theorem}

\begin{proof} By the hypothesis, we take a set of orthonormal vectors as follows:
\begin{eqnarray*}
X_1 & = & \frac{1}{\sqrt{2}}(\underbrace{\frac{1}{\sqrt{s}},\ldots ,\frac{1}{\sqrt{s}}}_{s}, \underbrace{\frac{1}{\sqrt{t}},\ldots ,\frac{1}{\sqrt{t}}}_{t})\\
X_2 & = & \sqrt{\frac{s}{s+t}}(\underbrace{1, 0, \ldots, 0}_{s}, \underbrace{-\sqrt{\frac{t}{s}}, 0, \ldots, 0}_{t})
\end{eqnarray*}
\begin{eqnarray*}
X_3 & = & \sqrt{\frac{s}{s+t}}(\underbrace{0, 1, \ldots, 0}_{s}, \underbrace{0, -\sqrt{\frac{t}{s}}, 0, \ldots, 0}_{t}), \cdots,\\
X_{\beta+1} & = & \sqrt{\frac{s}{s+t}}(\underbrace{0, \ldots, 0, 1, 0, \ldots, 0}_{s}, \underbrace{0, \ldots, 0, -\sqrt{\frac{t}{s}}, 0, \ldots, 0}_{t}).
\end{eqnarray*}
By Lemma \ref{le2,7}, we have
\begin{eqnarray*}
S_k(A_{\alpha}(G)) & = & \max\sum\limits_{i=1}^{k}\langle A_{\alpha}(G)X_i, X_i\rangle\\
& \geq & \sum\limits_{i=1}^{k}\langle A_{\alpha}(G)X_i, X_i\rangle\\
& = & \sum\limits_{i=1}^{k} \sum\limits_{uv\in E(G)}(\alpha x_u^2+2(1-\alpha)x_ux_v+\alpha x_v^2)\\
& = & m \left(\frac{\alpha}{2s}+2(1-\alpha)\frac{1}{2\sqrt{st}}+\frac{\alpha}{2t}\right)\\
& & +(k-1)\left(\frac{\alpha s}{s+t}+2(1-\alpha)\frac{-\sqrt{st}}{s+t}+\frac{\alpha t}{s+t}\right)\\
& = & \frac{\alpha m}{2}\left(\frac{1}{s}+\frac{1}{t}\right)+\frac{(1-\alpha)m}{\sqrt{st}}+(k-1)\left(\alpha -\frac{2(1-\alpha)\sqrt{st}}{s+t}\right)
\end{eqnarray*}
for $1\leq k \leq \beta+1$. The proof is completed. $\Box$
\end{proof}

Let $M$ be a real symmetric partitioned matrix of order $n$ described in the following block form
\[  \begin{pmatrix}
M_{11} &  \cdots & M_{1t} \\
\vdots & \ddots & \vdots \\
M_{t1} &  \cdots & M_{tt}
\end{pmatrix}, \]
where the diagonal blocks $M_{ii}$ are $n_i\times n_i$ matrices for any $i \in \{1,2,\ldots,t\}$ and $n=n_1+\cdots+n_t$.
For any $i, j \in \{1,2,\ldots,t\}$, let $b_{ij}$ denote the average row
sum of $M_{ij}$ , i.e. $b_{ij}$ is the sum of all entries in $M_{ij}$ divided by the number of
rows. Then $\mathcal{B}(M)=(b_{ij})$ (simply by $\mathcal{B}$) is called the quotient matrix of $M$.

\begin{lemma}{\bf (\cite{H})}\label{le5,1} 
Let $M$ be a symmetric partitioned matrix of order $n$ with eigenvalues $\xi_1\geq \xi_2 \geq \cdots  \geq \xi_n$, and let $\mathcal{B}$
its quotient matrix with eigenvalues $\eta_1\geq \eta_2 \geq \cdots  \geq \eta_r$ and $n>r$. Then $\xi_i \geq \eta_i \geq \xi_{n-r+i}$
for $i=1, 2, \ldots, r$.
\end{lemma}

\begin{corollary}\label{cor5,2} 
Let $M$ be a symmetric partitioned matrix of order $n$, and let $\mathcal{B}$ be its quotient matrix of order $k$. Then
$$S_{k}(M)\geq S_k(\mathcal{B}).$$
\end{corollary}

Let $\mathcal{B}$ be the quotient matrix of $A_{\alpha}(G)$ corresponding to the partition for the color classes of $G$. Then the following corollary is immediate.

\begin{corollary}\label{cor5,3} 
Let $G$ be a connected graph with $n$ vertices, $m$ edges, chromatic number $\chi$ and independence number $\theta$. If $0\leq \alpha < 1$, then
$$S_{\chi}(A_{\alpha}(G))\geq \frac{2\alpha m}{\theta}.$$
\end{corollary}

Let $U\subseteq V(G)$, $W\subseteq V(G)$ and $\partial(U, W)$ be the set of edges which connect vertices in $U$ with vertices in $W$.

\begin{theorem}\label{th5,4} 
 Let $0\leq \alpha < 1$ and $G$ be a connected graph with $n$ vertices and $m$ edges. For any given vertices subset $U=\{u_1,\ldots,u_{k-1}\}$ with $1\leq k \leq n$,
$$S_{k}(A_{\alpha}(G)) \geq \left(\alpha-\frac{1}{n-k+1}\right)\sum\limits_{u\in U}d_u+\frac{2m-(1-\alpha)|\partial(U, V(G)\backslash U)|}{n-k+1}.$$
\end{theorem}

\begin{proof} If $2\leq k \leq n$, then the quotient matrix of $A_{\alpha}(G)$ corresponding to the partition $V(G)=U\cup (V(G)\backslash U)$ of $G$ is
$$
	\mathcal{B}(G)=\left[
	\begin{array}{ccc|c}
	 & & &b_{1,k}\\
	 & A_{\alpha}(U) & & \vdots\\
	& & & b_{k-1,k}\\ \hline
	b_{k,1} &\cdots &b_{k,k-1} &b_{k,k}\\
	\end{array}
	\right],
	$$
where $A_{\alpha}(U)$ is the principal submatrix of $A_{\alpha}(G)$. By Lemma \ref{le5,1}, we have
\begin{eqnarray*}
S_{k}(A_{\alpha}(G)) & \geq & S_{k}(\mathcal{B}(G))\\
& = & tr(A_{\alpha}(U))+b_{k,k}\\
& = & \alpha\sum\limits_{u\in U}d(u)+\frac{2m-\sum\limits_{u\in U}d_u-(1-\alpha)|\partial(U, V(G)\backslash U)|}{n-k+1}\\
& = & \left(\alpha-\frac{1}{n-k+1}\right)\sum\limits_{u\in U}d_u+\frac{2m-(1-\alpha)|\partial(U, V(G)\backslash U)|}{n-k+1}.
\end{eqnarray*}

If $k=1$, then $U$ is an empty set. Thus $\sum\limits_{u\in U}d_u=0$ and $|\partial(U, V(G)\backslash U)|=0$. Taking a $n$-vector $X=(1, \ldots, 1)$, by Rayleigh's principle, we have
$$S_1(A_{\alpha}(G))=\lambda_1(A_{\alpha}(G)) \geq \frac{2m}{n}.$$
Therefore, the above inequality still holds for $k=1$.
This completes the proof. $\Box$
\end{proof}

If $U$ is a subset of a maximum independent set of $G$, by Theorem \ref{th5,4}, we have

\begin{corollary}\label{cor5,4} 
Let $G$ be a connected graph with $n$ vertices, $m$ edges and independence number $\theta$. If $0\leq \alpha < 1$, then
$$S_{k}(A_{\alpha}(G)) \geq \alpha(k-1)\delta+\frac{2m-(2-\alpha)(k-1)\delta}{n-k+1}$$
for $1\leq k \leq \theta+1$.
\end{corollary}

The next theorem is concerned with Problem \ref{pro1,1}. For $k=2$, we will prove that path is the minimum $S_k(A_{\alpha}(G))$ among all connected graphs for $\frac{1}{2}\leq \alpha< 1$. Let the sequence $(d_1, d_2, \ldots, d_n)$ be the set of a graph with the same degree sequence.

\begin{theorem}\label{th5,5} 
Let $G$ be a connected graph with $n\geq 12$ vertices. If $ \frac{1}{2}\leq \alpha < 1$, then
$$S_{2}(A_{\alpha}(G))\geq S_{2}(A_{\alpha}(P_n))$$
with equality if and only if $G=P_n$.
\end{theorem}

\begin{proof} By Corollary \ref{cor4,2}, we have $S_2(A_{\alpha}(P_n))< 4$. Let $T_n$ be a spanning tree of a connected graph $G$ with $n$ vertices. By Lemmas \ref{le2,8} and \ref{le2,13}, we have
$$S_2(A_{\alpha}(G))\geq S_2(A_{\alpha}(T_n)) \geq S_2(A_{1/2}(T_n))\geq S_2(A_{1/2}(T_{12}))$$
for $\frac{1}{2}\leq \alpha <1$ and $n\geq 12$.

In the following, we only need to show $S_2(A_{1/2}(T_{12}))\geq S_2(A_{\alpha}(P_n))$ for $\frac{1}{2}\leq \alpha <1$ and $n\geq 12$. For $\alpha=\frac{1}{2}$, we have $A_{1/2}(G)=\frac{1}{2}Q(G)$. From Theorem 3.1 in \cite{OL}, we know that $S_2(Q(G))\geq S_2(D(G))+1$ with equality if and only if $G$ is the star $K_{1,\,n-1}$ or the complete graph $K_3$. Let $\Delta_2(G)$ be the second largest degree of a graph $G$.

If $\Delta(T_{12})\geq 4$ and $\Delta_2(T_{12})\geq 3$, then we have
$$S_2(A_{1/2}(T_{12}))\geq \frac{1}{2}(S_2(D(G))+1)=\frac{1}{2}(4+3+1)= 4> S_2(A_{\alpha}(P_n))$$
for $\frac{1}{2}\leq \alpha <1$ and $n\geq 12$.

If $\Delta(T_{12})\geq 5$ and $\Delta_2(T_{12})\geq 2$, then we have
$$S_2(A_{1/2}(T_{12}))\geq \frac{1}{2}(S_2(D(G))+1)=\frac{1}{2}(5+2+1)= 4> S_2(A_{\alpha}(P_n))$$
for $\frac{1}{2}\leq \alpha <1$ and $n\geq 12$.

If $\Delta(T_{12})= 4$ and $\Delta_2(T_{12})=2$, then $T_{12}$ is one of the trees $(4, 2, 2, 2, 2, 2, 2, 2, 1, 1, 1, 1)$.  By computation with computer, we have
$$S_2(A_{1/2}(T_{12}))\geq S_2(A_{1/2}(T'))\geq \frac{1}{2}\times 8.57037=4.285185> S_2(A_{\alpha}(P_n))$$
for $\frac{1}{2}\leq \alpha <1$ and $n\geq 12$, where $T'$, shown in Fig. 4.1, is a tree with minimum sum of the two largest $A_{1/2}$-eigenvalues in the set of trees $(4, 2, 2, 2, 2, 2, 2, 2, 1, 1, 1, 1)$.

If $\Delta(T_{12})=\Delta_2(T_{12})=3$, then we assume that $x$, $y$ and $z$ be the number of the vertices of the degree three, the degree two and the degree one, respectively. Thus, we have
$$x+y+z=12, \qquad 3x+2y+z=22.$$
Solve the above equations, we get
$x=2, y=6, z=4$; $x=3, y=4, z=5$; $x=4, y=2, z=6$; $x=5, y=0, z=7$. Further, we know that $T_{12}$ is one of the trees $(3, 3, 2, 2, 2, 2, 2, 2, 1, 1, 1, 1)$, $(3, 3, 3, 2, 2, 2, 2, 1, 1, 1, 1, 1)$, $(3, 3, 3, 3, 2, 2, 1, 1, 1, 1, 1, 1)$ and $(3, 3, 3, 3, 3, 1, 1, 1, 1, 1, 1, 1)$. By computation with computer, we have
$$S_2(A_{1/2}(T_{12}))\geq S_2(A_{1/2}(T''))\geq \frac{1}{2}\times 8.31903=4.159515> S_2(A_{\alpha}(P_n))$$
for $\frac{1}{2}\leq \alpha <1$ and $n\geq 12$, where $T''$, shown in Fig. 4.1, is a tree with minimum sum of the two largest $A_{1/2}$-eigenvalues in the set of trees $(3, 3, 2, 2, 1, 1, 1, 1, 1, 1, 1, 1)$, $(3, 3, 3, 2, 2, 2, 2, 1, 1, 1, 1, 1)$, $(3, 3, 3, 3, 2, 2, 1, 1, 1, 1, 1, 1)$ and $(3, 3, 3, 3, 3, 1, 1, 1, 1, 1, 1, 1)$.

If $\Delta(T_{12})=3$ and $\Delta_2(T_{12})=2$, then $T_{12}$ is one of the trees $(3, 2, 2, 2, 2, 2, 2, 2, 2, 1, 1, 1)$. By computation with computer, we have
$$S_2(A_{1/2}(T_{12}))\geq S_2(A_{1/2}(T'''))\geq \frac{1}{2}\times 8.02294=4.01147> S_2(A_{\alpha}(P_n))$$
for $\frac{1}{2}\leq \alpha <1$ and $n\geq 12$, where $T'''$, shown in Fig. 4.1, is a tree with minimum sum of the two largest $A_{1/2}$-eigenvalues in the set of trees $(3, 2, 2, 2, 2, 2, 2, 2, 2, 1, 1, 1)$.

If $\Delta_2(T_{12})=1$, then $T_{12}$ is a star $K_{1,\, 11}$. Thus,
$$S_2(A_{1/2}(K_{1,\, 11}))\geq \frac{1}{2}\times 13=6.5>S_2(A_{\alpha}(P_n))$$
for $\alpha \in [\frac{1}{2}, 1)$ and $n\geq 12$.

Combining the above argument, we have $S_2(A_{1/2}(T_{12}))\geq S_2(A_{\alpha}(P_n))$ for $\alpha \in [\frac{1}{2}, 1)$ and $n\geq 12$. Further, we get $S_{2}(A_{\alpha}(G))\geq S_{2}(A_{\alpha}(P_n))$ for $\frac{1}{2}\leq \alpha <1$ and $n\geq 12$, and equality holds if and only if $G=P_n$, completing the proof. $\Box$

\end{proof}

\begin{picture}(300,65)

    \put(70,40){\circle*{3}}
       \put(70,40){\line(5,4){23}}
       \put(70,40){\line(5,-4){23}}
       \put(70,40){\line(-5,4){23}}
       \put(70,40){\line(-5,-4){23}}
       \put(92,58){\circle*{3}}
       \put(112,58){\circle*{3}}
       \put(92,58){\line(1,0){20}}
        \put(132,58){\circle*{3}}
       \put(112,58){\line(1,0){20}}
       \put(92,22){\circle*{3}}
       \put(112,22){\circle*{3}}
       \put(92,22){\line(1,0){20}}
        \put(132,22){\circle*{3}}
       \put(112,22){\line(1,0){20}}
       \put(48,58){\circle*{3}}
       \put(28,58){\circle*{3}}
       \put(28,58){\line(1,0){20}}
        \put(48,22){\circle*{3}}
        \put(28,22){\circle*{3}}
       \put(28,22){\line(1,0){20}}
        \put(8,22){\circle*{3}}
       \put(8,22){\line(1,0){20}}
        \put(68,-5){\footnotesize $T'$}

        \put(185,40){\circle*{3}}
       \put(185,40){\line(-5,4){23}}
       \put(185,40){\line(-5,-4){23}}
       \put(185,40){\line(1,0){20}}
        \put(205,40){\circle*{3}}
       \put(205,40){\line(5,4){23}}
       \put(205,40){\line(5,-4){23}}
       \put(227,58){\circle*{3}}
       \put(247,58){\circle*{3}}
       \put(227,58){\line(1,0){20}}
        \put(267,58){\circle*{3}}
        \put(247,58){\line(1,0){20}}
        \put(287,58){\circle*{3}}
        \put(267,58){\line(1,0){20}}
       \put(227,22){\circle*{3}}
       \put(247,22){\circle*{3}}
       \put(227,22){\line(1,0){20}}
        \put(267,22){\circle*{3}}
       \put(247,22){\line(1,0){20}}
       \put(287,22){\circle*{3}}
       \put(267,22){\line(1,0){20}}
       \put(163,58){\circle*{3}}
        \put(163,22){\circle*{3}}
         \put(213,-5){\footnotesize $T''$}

       \put(320,40){\circle*{3}}
       \put(320,40){\line(5,4){23}}
       \put(320,40){\line(5,-4){23}}
       \put(340,40){\circle*{3}}
       \put(320,40){\line(1,0){20}}
       \put(360,40){\circle*{3}}
       \put(340,40){\line(1,0){20}}
        \put(380,40){\circle*{3}}
       \put(360,40){\line(1,0){20}}
        \put(342,58){\circle*{3}}
       \put(362,58){\circle*{3}}
       \put(342,58){\line(1,0){20}}
        \put(382,58){\circle*{3}}
        \put(362,58){\line(1,0){20}}
        \put(402,58){\circle*{3}}
        \put(382,58){\line(1,0){20}}
       \put(342,22){\circle*{3}}
       \put(362,22){\circle*{3}}
       \put(342,22){\line(1,0){20}}
        \put(382,22){\circle*{3}}
       \put(362,22){\line(1,0){20}}
       \put(402,22){\circle*{3}}
       \put(382,22){\line(1,0){20}}
       \put(358,-5){\footnotesize $T'''$}

\put(145,-18){Fig. 4.1 \quad Trees $T'$, $T''$, $T'''$. }
\end{picture}

\vskip 8mm

\begin{problem}\label{pro5,1} 
For $0\leq \alpha <\frac{1}{2}$, which graph(s) minimize the sum of the two largest of $A_{\alpha}$-eigenvalues among all connected graphs with $n$ vertices?
\end{problem}

\section{\large On the sum of the largest $A_{\alpha}$-eigenvalues of graph operations}

\begin{theorem}\label{th6,1} 
Let $G$ be a graph with $n$ vertices. If $0\leq \alpha \leq 1$, then
$$(1-\alpha)n+(\alpha n-1)k \leq S_{k}(A_{\alpha}(G))+S_{k}(A_{\alpha}(\overline{G}))\leq k[(2-\alpha)n+\alpha(\Delta-\delta-1)-(1-\alpha)(k+1)].$$
\end{theorem}

\begin{proof} From Proposition 36 in \cite{N}, we have $S_{k}(A_{\alpha}(K_n))=(1-\alpha)n+(\alpha n-1)k$. Since $A_{\alpha}(G)+A_{\alpha}(\overline{G})=A_{\alpha}(K_n)$, by Theorem \ref{th1,1}, we have
$$S_{k}(A_{\alpha}(G))+S_{k}(A_{\alpha}(\overline{G}))\geq S_{k}(A_{\alpha}(K_n))=(1-\alpha)n+(\alpha n-1)k.$$
By Lemma \ref{le2,9}, we have
$$S_{k}(A_{\alpha}(G))\leq \alpha(d_1+d_2+\cdots+d_k)+(1-\alpha)\left(kn-\frac{k(k+1)}{2}\right).$$
Thus
\begin{eqnarray*}
S_{k}(A_{\alpha}(G))+S_{k}(A_{\alpha}(\overline{G})) & \leq & \alpha k(n-1)+\alpha\sum\limits_{i=1}^{k}(d_i-d_{n-i+1})+(1-\alpha)(2kn-k(k+1))\\
& \leq & \alpha k(n-1)+\alpha k (\Delta-\delta)+(1-\alpha)(2kn-k(k+1))
\end{eqnarray*}
\begin{eqnarray*}
& = & k[(2-\alpha)n+\alpha(\Delta-\delta-1)-(1-\alpha)(k+1)].
\end{eqnarray*}
This completes the proof. $\Box$
\end{proof}

\begin{theorem}\label{th6,2} 
Let $G$ be a graph with $n$ vertices and $m\geq 1$ edges. Then
$$S_{k}(A_{\alpha}(\mathcal{L}(G)))\leq 2k(\alpha \Delta-1)+(1-\alpha)S_{k}(Q(G))$$
for $1\leq k \leq s$, where $s=\min\{n,m\}$. If $m>n$, then
$$S_{k}(A_{\alpha}(\mathcal{L}(G)))\leq 2\alpha k(\Delta-1)+2(1-\alpha)(m-k)$$
for $n+1\leq k \leq m$.
\end{theorem}

\begin{proof} If a vertex $w$ is in one-to-one correspondence with the edge $uv$ of the graph $G$, then $d_{\mathcal{L}(G)}(w)=d_G(u)+d_G(v)-2$. By Theorem \ref{th1,1} and Lemma \ref{le2,10}, we have
\begin{eqnarray*}
S_{k}(A_{\alpha}(\mathcal{L}(G))) & \leq & \alpha S_{k}(D(\mathcal{L}(G)))+(1-\alpha)S_{k}(A(\mathcal{L}(G)))\\
& \leq & \alpha k(2\Delta-2)+(1-\alpha)(S_{k}(Q(G))-2k)\\
& = & 2k(\alpha \Delta-1)+(1-\alpha)S_{k}(Q(G))
\end{eqnarray*}
for $1\leq k \leq s$, where $s=\min\{n,m\}$. If $m>n$, then we have
$$S_{k}(A_{\alpha}(\mathcal{L}(G)))\leq \alpha k(2\Delta-2)+(1-\alpha)(2m-2n-2(k-n))=2\alpha k(\Delta-1)+2(1-\alpha)(m-k)$$
for $n+1\leq k \leq m$. This completes the proof. $\Box$
\end{proof}

By Lemma \ref{le2,6} and Conjecture \ref{co1,2}, we have

\begin{corollary}\label{cor6,1} 
If $T$ is a tree with $n$ vertices, then $S_{k}(A_{\alpha}(\mathcal{L}(T)))\leq 2k\alpha(\Delta-1)+(1-\alpha)(n-2)$ for $1\leq k \leq n-1$. If $U$ is a unicyclic graph with $n$ vertices, then $S_{k}(A_{\alpha}(\mathcal{L}(U)))\leq 2k(\alpha \Delta-1)+(1-\alpha)(n+\frac{k^2+k}{2})$ for $1\leq k \leq n$. If $B$ is a bicyclic graph with $n$ vertices, then $S_{k}(A_{\alpha}(\mathcal{L}(B)))\leq 2k(\alpha \Delta-1)+(1-\alpha)(n+1+\frac{k^2+k}{2})$ for $1\leq k \leq n$.
\end{corollary}

\begin{theorem}\label{th6,3} 
Let $G$ be a $K_3$-free and $C_4$-free graph with $n$ vertices and $m$ edges. If $0\leq \alpha \leq 1$, then
$$S_{k}(A_{\alpha}(G^2))\leq \alpha(Z_1(G)-(n-k)\delta^2(G))+(1-\alpha)\left(2m-\frac{1}{n-k}S_{k}^2(A(G))-(k-1)a(G)\right).$$
\end{theorem}

\begin{proof} Since $\sum\limits_{i=1}^{n}\lambda_i(A(G))=0$ and $\sum\limits_{i=1}^{n}\lambda_i^2(A(G))=2m$, by the Cauchy-Schwarz inequality, we have
\begin{eqnarray*}
S_{k}(A^2(G)) & = & \lambda_1^2(A(G))+\lambda_2^2(A(G))+\cdots+\lambda_k^2(A(G))\\
& = & 2m-\sum\limits_{i=k+1}^{n}\lambda_i^2(A(G))\\
& \leq & 2m-\frac{1}{n-k}\left(\sum\limits_{i=k+1}^{n}\lambda_i(A(G))\right)^2\\
& = & 2m-\frac{1}{n-k}\left(\sum\limits_{i=1}^{k}\lambda_i(A(G))\right)^2\\
& = & 2m-\frac{1}{n-k}S_{k}^2(A(G)).
\end{eqnarray*}
Since $\sum\limits_{u \in V(G^2)}d_u=Z_1(G)$, by Theorem \ref{th1,1} and Lemma \ref{le2,11}, we have
\begin{eqnarray*}
S_{k}(A_{\alpha}(G^2)) & \leq & \alpha S_{k}(D(G^2))+(1-\alpha)S_{k}(A(G^2))\\
& \leq & \alpha S_{k}(D(G^2))+(1-\alpha)(S_{k}(A^2(G))+S_{k}(-L(G)))\\
& \leq & \alpha(Z_1(G)-(n-k)\delta^2(G))\\
& & +(1-\alpha)\left(2m-\frac{1}{n-k}S_{k}^2(A(G))-(k-1)a(G)\right).
\end{eqnarray*}
This completes the proof. $\Box$

\end{proof}

\begin{theorem}\label{th6,4} 
Let $G$ be a graph with $n$ vertices. If $0\leq \alpha \leq 1$, then
$$S_{k}(A_{\alpha}(\mathcal{D}(G)))\leq \begin{dcases}
4\sum\limits_{i=1}^{k/2}d_i(G)+2(1-\alpha)S_k(A(G)), & \text{if}\,\, 1 < k< n \,\,\text{is} \,\,\text{even};\\
4\sum\limits_{i=1}^{(k-1)/2}d_i(G)+2d_{(k+1)/2}+2(1-\alpha)S_k(A(G)), & \text{if}\,\,1 \leq k< n \,\, \text{is} \,\,\text{odd};\\
4\sum\limits_{i=1}^{k/2}d_i(G), & \text{if}\,\, n \leq k\leq 2n \,\,\text{is} \,\,\text{even};\\
4\sum\limits_{i=1}^{(k-1)/2}d_i(G)+2d_{(k+1)/2}, & \text{if}\,\,n \leq k\leq 2n \,\, \text{is} \,\,\text{odd},
\end{dcases}$$
where $d_1(G)\geq d_2(G)\geq \cdots \geq d_n(G)$ is the vertex degree sequence of $G$.
\end{theorem}

\begin{proof} By the definition of $\mathcal{D}(G)$, the $A_{\alpha}$-matrix of the double graph of $G$ is
\begin{eqnarray*}
A_{\alpha}(\mathcal{D}(G)) & = & \alpha D(\mathcal{D}(G))+(1-\alpha)A(\mathcal{D}(G))\\
& = & \alpha \begin{pmatrix}
2 &  0 \\
0 &  2
\end{pmatrix}\otimes D(G)+(1-\alpha)\begin{pmatrix}
1 &  1 \\
1 &  1
\end{pmatrix}\otimes A(G),
\end{eqnarray*}
where $M\otimes N$ is the Kronecker product (or tensor product) of $M$ and $N$. Thus the spectrum of $\begin{pmatrix}
2 &  0 \\
0 &  2
\end{pmatrix}\otimes D(G)$  and $\begin{pmatrix}
1 &  1 \\
1 &  1
\end{pmatrix}\otimes A(G)$ are
$$ 2d_1(G),  2d_1(G),  2d_2(G), 2d_2(G), \ldots, 2d_n(G), 2d_n(G)$$
and
$$2\lambda_1(A), 2\lambda_2(A), \ldots, 2\lambda_n(A), 0, 0, \ldots, 0, $$
respectively. By Theorem \ref{th1,1}, we have
\begin{eqnarray*}
S_k(A_{\alpha}(\mathcal{D}(G))) & \leq & \alpha S_k(D(\mathcal{D}(G)))+(1-\alpha)S_k(A(\mathcal{D}(G)))\\
& = & \begin{dcases}
4\sum\limits_{i=1}^{k/2}d_i(G)+2(1-\alpha)S_k(A(G)), & \text{if}\,\, 1 < k< n \,\,\text{is} \,\,\text{even};\\
4\sum\limits_{i=1}^{(k-1)/2}d_i(G)+2d_{(k+1)/2}+2(1-\alpha)S_k(A(G)), & \text{if}\,\,1 \leq k< n \,\, \text{is} \,\,\text{odd};\\
4\sum\limits_{i=1}^{k/2}d_i(G), & \text{if}\,\, n \leq k\leq 2n \,\,\text{is} \,\,\text{even};\\
4\sum\limits_{i=1}^{(k-1)/2}d_i(G)+2d_{(k+1)/2}, & \text{if}\,\,n \leq k\leq 2n \,\, \text{is} \,\,\text{odd}.
\end{dcases}
\end{eqnarray*}
This completes the proof. $\Box$
\end{proof}

\small {

}


\begin{thebibliography}{99}

\bibitem{ACGMR} N. Abreu, D.M. Cardoso, I. Gutman, E.A. Martins, M. Robbiano, Bounds for the signless Laplacian energy, Linear Algebra Appl. 435 (2011) 2365-2374.%

\bibitem{ALOLA} B. Amaro, L. de Lima, C.S. Oliveira, C. Lavor, N. Abreu, A note on the sum of the largest signless Laplacian eigenvalues, Electron. Notes Discrete Math. 54 (2016) 175-180.%

\bibitem{AOT} F. Ashraf, G.R. Omidi, B. Tayfeh-Rezaie, On the sum of signless Laplacian eigenvalues of graphs, Linear Algebra Appl. 438 (2013) 4539-4546.%

\bibitem{B} R. Bhatia, Matrix Analysis, GTM 169, Springer-Verlag, New York, 1997.%

\bibitem{BDFG} B. Borovi\'{c}anin, K.Ch. Das, B. Furtula, I. Gutman, Bounds for Zagreb indices, MATCH Commun. Math. Comput. Chem. 78 (2017) 17-100.%

\bibitem{CG} D. Cvetkovi\'{c}, I. Gutman, The algebraic multiplicity of the number zero in the spectrum of a bipartite graph, Mat. Vesnik 9 (1972) 141-150.%

\bibitem{CHJL} X. Chen, G. Hao, D. Jin, J. Li, Note on a conjecture for the sum of signless Laplacian eigenvalues, Czechoslovak Math. J. 68 (2018) 601-610.%

\bibitem{CLM} Y. Chen, D. Li, J. Meng, On the second largest $A_{\alpha}$-eigenvalues of graphs, Linear Algebra Appl. 580 (2019) 343-358.%

\bibitem{CLWM} Y. Chen, D. Li, Z. Wang, J. Meng, $A_{\alpha}$-spectral radius of the second power of a graph, Appl. Math. Comput. 359 (2019) 418-425.%

\bibitem{CRS} D. Cvetkovi\'{c}, P. Rowlinson, S.K. Simi\'{c}, Eigenvalue bounds for the signless Laplacian, Publ. Inst. Math. (Beograd) (N.S.) 81 (2007) 11-27.%

\bibitem{CS} D. Cvetkovi\'{c}, S.K. Simi\'{c}, Towards a spectral theory of graphs based on the signless Laplacian. I, Publ. Inst. Math. (Beograd) (N.S.)
85 (2009) 19-33.%

\bibitem{D} Z. Du, The sum of the first two largest signless Laplacian eigenvalues of trees and unicyclic graphs, Electron. J. Linear Algebra 35 (2019) 449-467.%

\bibitem{DMS} K.Ch. Das, S.A. Mojallal, S. Sun, On the sum of the $k$ largest eigenvalues of graphs and maximal energy of bipartite graphs, Linear Algebra Appl. 569 (2019) 175-194.%

\bibitem{EMNA} J. Ebrahimi B, B. Mohar, V. Nikiforov, A.S. Ahmady, On the sum of two largest eigenvalues of a symmetric matrix, Linear Algebra Appl. 429 (2008) 2781-2787.%

\bibitem{F} K. Fan, On a theorem of Weyl concerning eigenvalues of linear transformations I, Proc. Nat. Acad. Sci. USA 35 (1949) 652-655.%

\bibitem{FT} G.H. Fath-Tabar, Old and new Zagreb indices of graphs, MATCH Commun. Math. Comput. Chem. 65 (2011) 79-84.%

\bibitem{G} D. Gernert, private communication, see also $<$http://www.sgt.pep.ufrj.br/home-arquivos/prob-abertos.html$>$.%

\bibitem{GCP} H.A. Ganie, B.A. Chat, S. Pirzada, Signless Laplacian energy of a graph and energy of a line graph, Linear Algebra Appl. 544 (2018) 306-324.%

\bibitem{GZ} H. Guo, B. Zhou, On the $\alpha$-spectral radius of graphs, Appl. Anal. Discrete Math. 14 (2020) 431-458.%

\bibitem{H} W.H. Haemers, Interlacing eigenvalues and graphs, Linear Algebra Appl. 226-228 (1995) 593-616.%

\bibitem{HLX} X. Huang, H. Lin, J. Xue, The Nordhaus-Gaddum type inequalities of $A_{\alpha}$-matrix, Appl. Math. Comput. 365 (2020) 124716.%

\bibitem{HMT} W.H. Haemers, A. Mohammadian, B. Tayfeh-Rezaie, On the sum of Laplacian eigenvalues of graphs, Linear Algebra Appl. 432 (2010) 2214-2221.%

\bibitem{LDS} S. Liu, K.Ch. Das, J. Shu, On the eigenvalues of $A_{\alpha}$-matrix of graphs, Discrete Math. 343 (2020) 111917.%

\bibitem{LDSS} S. Liu, K.Ch. Das, S. Sun, J. Shu, On the least eigenvalue of $A_{\alpha}$-matrix of graphs, Linear Algebra Appl. 586 (2020) 347-376.%

\bibitem{LGZ} H. Lin, H. Guo, B. Zhou, On the $\alpha$-spectral radius of irregular uniform hypergraphs, Linear Multilinear Algebra 68 (2020) 265-277.%

\bibitem{LL} X. Liu, S. Liu, On the $A_{\alpha}$-characteristic polynomial of a graph, Linear Algebra Appl. 546 (2018) 274-288.%

\bibitem{LLX} H. Lin, X. Liu, J. Xue, Graphs determined by their $A_{\alpha}$-spectra, Discrete Math. 342 (2019) 441-450.%

\bibitem{LMG} Z. Lin, L. Miao, S. Guo, The $A_{\alpha}$-spread of a graph, Linear Algebra Appl. 606 (2020) 1-22.%

\bibitem{LSG} X. Li, Y. Shi, I. Gutman, Graph Energy, Springer, New York, 2012.%

\bibitem{LW} S. Li, W. Wei, The multiplicity of an $A_{\alpha}$-eigenvalue: A unified approach for mixed graphs and complex unit gain graphs, Discrete Math. 343 (2020) 111916.%

\bibitem{LWCL} J. Liu, X. Wu, J. Chen, B. Liu, The $A_{\alpha}$ spectral radius characterization of some digraphs, Linear Algebra Appl. 563 (2019) 63-74.%

\bibitem{LXS} H. Lin, J. Xue, J. Shu, On the $A_{\alpha}$-spectra of graphs, Linear Algebra Appl. 556 (2018) 210-219.%

\bibitem{M} B. Mohar, On the sum of $k$ largest eigenvalues of graphs and symmetric matrices, J. Combin. Theory Ser. B 99 (2009) 306-313.%

\bibitem{N} V. Nikiforov, Merging the $A$- and $Q$-spectral theories, Appl. Anal. Discrete Math. 11 (2017) 81-107.%

\bibitem{N1} V. Nikiforov, On the sum of $k$ largest singular values of graphs and matrices, Linear Algebra Appl. 435 (2011) 2394-2401.

\bibitem{N2} V. Nikiforov, Linear combinations of graph eigenvalues, Electron. J. Linear Algebra 15 (2006) 329-336.%

\bibitem{NPRS} V. Nikiforov, G. Past\'{e}n, O. Rojo, R.L. Soto, On the $A_{\alpha}$-spectra of trees, Linear Algebra Appl. 520 (2017) 286-305.%

\bibitem{NR} V. Nikiforov, O. Rojo, A note on the positive semidefiniteness of $A_{\alpha}(G)$, Linear Algebra Appl. 519 (2017) 156-163.%

\bibitem{OL} C.S. Oliveira, L. de Lima, A lower bound for the sum of the two largest signless Laplacian eigenvalues, Electron. Notes Discrete Math. 55 (2016) 173-176.%

\bibitem{OLRC} C.S. Oliveira, L. de Lima, P. Rama, P. Carvalho, Extremal graphs for the sum of the two largest signless Laplacian eigenvalues, Electron. J. Linear Algebra 30 (2015) 605-612.%

\bibitem{RSR} O. Rojo, R. Soto, H. Rojo, Bounds for sums of eigenvalues and applications, Comput. Math. Appl. 39 (2000) 1-15.%

\bibitem{YY} J. Yang, L. You, On a conjecture for the signless Laplacian eigenvalues, Linear Algebra Appl. 446 (2014) 115-132.%



\end{thebibliography}
\end{document}